\documentclass[12pt,oneside]{amsart}

\usepackage[english]{babel}
\usepackage{cite}
\usepackage{latexsym}
\usepackage{amssymb,amsthm,amsmath}
\usepackage{hyperref}
\usepackage{graphicx}
\usepackage{longtable}
\usepackage{xcolor}
\usepackage{enumerate}
\usepackage{cancel}

\newtheorem{theorem}{Theorem}[section]
\newtheorem{proposition}[theorem]{Proposition}
\newtheorem{lemma}[theorem]{Lemma}

\newtheorem{remark}[theorem]{Remark}

{\theoremstyle{definition}
\newtheorem*{definition*}{Definition}
\newtheorem{definition}[theorem]{Definition}}
\newtheorem*{proposition*}{Proposition}
\newtheorem*{corollary*}{Corollary}
\newtheorem*{lemma*}{Lemma}
\newtheorem*{remark*}{Remark}

\definecolor{dgreen}{rgb}{0.13,0.7,.63}

\def\daniele#1 {\fbox {\footnote {\ }}\ \footnotetext { From Daniele: {\color{red}#1}}}
\def\maria#1 {\fbox {\footnote {\ }}\ \footnotetext { From Maria: {\color{blue}#1}}}
\def\giovanni#1 {\fbox {\footnote {\ }}\ \footnotetext { From Giovanni: {\color{green}#1}}}

\title{Bent functions from triples of permutation polynomials}

\author{Daniele Bartoli}
\address{Dipartimento di Matematica e Informatica, Universit\`a degli Studi di Perugia, Perugia,  Italy}
\email {daniele.bartoli@unipg.it}

\author{Maria Montanucci}
\address{Department of Applied Mathematics and Computer Science, Technical University of Denmark, Kongens Lyngby, Denmark}
\email {marimo@dtu.dk}

\author{Giovanni Zini}
\address{Dipartimento di Matematica e Applicazioni, Universit\`a degli Studi di Milano-Bicocca, Milano, Italy}
\email {giovanni.zini@unimib.it}

\date{}

\begin{document}

\begin{abstract}
We provide constructions of bent functions using triples of permutations. This approach is due to Mesnager. In general,  involutions have been mostly considered in such a machinery; we provide some other suitable triples of permutations, using monomials, binomials,  trinomials, and quadrinomials. 
\end{abstract}

\maketitle

{\bf Keywords:} Permutation; Boolean function; bent function; monomial, binomial, trinomial. \\

\section{Introduction}
Bent functions are Boolean functions with an even number of variables which achieve the maximum possible nonlinearity. They have been introduced by  Rothaus \cite{Rothaus1976} and Dillon \cite{Dillon1974} in the Seventies and they have interesting applications in Combinatorics and Design Theory (difference sets and symmetric designs), Coding Theory (several types of codes such as Reed-Muller codes \cite{14} and  two-weight codes \cite{2}), Cryptography \cite{8},  Sequence Theory \cite{30}, Graph Theory \cite{31}. 

A complete classification of bent functions seems to be hopeless, although in the last years explicit constructions of such functions have been provided, using different techniques. For a more detailed introduction on this topic we refer to \cite{CM2016,Mesnager2016_b,Carlet2006,Carlet2014,CM2011,Kyureghyan2011} and the references therein.

Some recent constructions of bent functions rely on an approach due to Mesnager \cite{Mesnager2014}, where triples of permutations satisfying particular constraints are needed; see Definition \ref{Def:Phi}. In general,  involutions have been mostly considered in such a machinery; see \cite{CM2018,MCM2015,CMS2016}. 

In this paper, we provide some other families of permutations which can be used to construct bent functions. In particular, we use triples of monomials which are not all involutions (see Theorem  \ref{Prop:Fam1}), monomials and binomials (see Theorem \ref{Prop:Fam2}), monomials and trinomials  (see Theorem \ref{Prop:Fam3}), binomials of order four (see Theorem \ref{Prop:Fam2}). 

\section{Bent functions and permutations}

Let $q=2^n$, $n\in \mathbb{N}$, and denote by $\mathbb{F}_{q}$ the finite field with $q$ elements. A \emph{boolean function} on $\mathbb{F}_q$ is a map from $\mathbb{F}_q$ to $\mathbb{F}_2$. On the one hand each boolean function can be represented as a polynomial $f(x)=\sum_{j=0}^{q-1} a_j x^j \in \mathbb{F}_q[x]$. On the other hand a polynomial $f(x)$ of this shape corresponds to a boolean function if and only if $a_0,a_{q-2}\in \mathbb{F}_2$ and $a_{2j\pmod {q-1}}=a_j^2$ for all $j \neq 0,q-1$. In this case $f(x)$ is the polynomial form of the boolean function; see \cite{8}.

In the case $n$ even, boolean functions can be described by a bivariate representation. Let $n=2m$ and identify $\mathbb{F}_{q}$ with $\mathbb{F}_{2^m}\times \mathbb{F}_{2^m}$. Each element of $\mathbb{F}_{q}$ can be uniquely described by a pair $(x,y)\in \mathbb{F}_{2^m}^2$. Then, each polynomial function $f$ over $\mathbb{F}_{q}$  can be  determined, by Lagrange interpolation, by a unique bivariate polynomial
$$\sum_{0\leq i, \ j \leq 2^m-1}a_{i,j}x^iy^j\in \mathbb{F}_{2^m}[x,y].$$ 

Each Boolean function can be written (non-uniquely) in the form 
$f(x,y)=Tr^{m}_1(P(x,y))$, where $P(x,y)\in \mathbb{F}_{2^m}[x,y]$ and $Tr^m_1$ denotes the trace function from $\mathbb{F}_{2^m}$ to $\mathbb{F}_2$ defined by $Tr^m_1(x)= x^{2^{m-1}}+x^{2^{m-2}}+\cdots+x^2+x$.

For a boolean function $f$ on $\mathbb{F}_{2^n}$, its Walsh-Hadamard transform $\widehat{\chi_f}:\mathbb{F}_{2^n}\rightarrow\mathbb{Z}$ is defined as follows:

$$
\widehat{\chi_f}(z)=\sum_{x \in \mathbb{F}_{2^n}} (-1)^{f(x)+Tr^n_1(zx)}.$$

\begin{definition}\label{Def:Bent}
Let $n\in \mathbb{N}$ be even. A Boolean function $f$ on $\mathbb{F}_{2^n}$ is called \emph{bent} if its Walsh-Hadamard transform satisfies $ \widehat{\chi_f}(z)=\pm 2^{n/2}$ for all $z \in \mathbb{F}_{2^n}$.
\end{definition}

Bent functions of the (bilinear) form $Tr^m_1(\Phi(y)x)+g(y)$, where $(x,y) \in \mathbb{F}_{2^m}\times \mathbb{F}_{2^m}$, $m\geq 1$, $\Phi$ is a permutation of $\mathbb{F}_{2^m}$ and $g$ is any Boolean function over $\mathbb{F}_{2^m}$, are called Maiorana-McFarland bent functions; see \cite{McFarland1973}.

In this paper we construct bent functions following the approach used in \cite{CM2018,MCM2015}. The ingredients are permutations of $\mathbb{F}_{2^n}$ satisfying the following property ($\mathcal{A}_n$).

\begin{definition}\cite{CM2018,Mesnager2016_a}\label{Def:Phi}
Let $n$ be a positive integer and $q=2^n$. Three permutations $\Phi_1$, $\Phi_2$ and $\Phi_3$ of $\mathbb{F}_q$ are said to satisfy ($\mathcal{A}_n$) if the following two conditions hold:
\begin{enumerate}
\item their sum $\Psi = \Phi_1+\Phi_2+\Phi_3$ is a permutation of $\mathbb{F}_q$;
\item $\Psi^{-1} = \Phi_1^{-1}+\Phi_2^{-1}+\Phi_3^{-1}$.
\end{enumerate}
\end{definition}

In Section \ref{Sec:Main} we will construct triples of permutations satisfying property ($\mathcal{A}_n$). 

The key result to construct bent functions from such triples is the following theorem.

\begin{theorem}\cite{Mesnager2014}\label{Th:Costruzione}
Let $n$ be a positive integer and $q=2^n$. Consider three permutations $\Phi_1$, $\Phi_2$, $\Phi_3$ of $\mathbb{F}_{q}$. Then the Boolean function
\begin{eqnarray*}
g(x, y) &=& Tr^n_1(x\Phi_1(y))Tr^n_1(x\Phi_2(y))\\
&&+Tr^n_1(x\Phi_2(y))Tr^n_1(x\Phi_3(y))\\
&&+Tr^n_1(x\Phi_1(y))Tr^n_1(x\Phi_3(y))
\end{eqnarray*}
 is bent over $\mathbb{F}_q\times \mathbb{F}_q$ if and only if $\Phi_1$, $\Phi_2$, $\Phi_3$ satisfy property ($\mathcal{A}_n$).
 \end{theorem}

The above theorem has been used to construct bent functions starting from monomial permutations \cite{Mesnager2014}, involutions \cite{Mesnager2016_c,CM2018}, or other families of permutations satisfying property ($\mathcal{A}_n$) \cite{Mesnager2016_a}. In this paper we construct bent functions starting from triples of permutations which are monomials (see Theorem \ref{Prop:Fam1}),  binomials and monomials (see Theorem \ref{Prop:Fam2}), monomials and trinomials  (see Theorem \ref{Prop:Fam3}), binomials and quadrinomials (see Theorem \ref{Prop:Fam4}), monomials and quadrinomials (see Theorem \ref{Prop:Fam5}).

A general problem, when dealing with bent functions, is the equivalence issue. For many bent functions so far constructed, the  equivalence question is still open. Also, if an answer is given,  the most common result is a proof that they are Extended Affine  equivalent to a Maiorana-McFarland bent function. The authors in \cite{CM2018} give a sufficient criterion for a triple $(\Phi_1,\Phi_2,\Phi_3)$ satisfying property $(\mathcal{A}_n)$ to yield a Maiorana-McFarland bent function; see \cite[Theorem 5]{CM2018}. For each pair $(i,j)$, $1\leq i<j\leq 3$, let us define 
\begin{equation}\label{Eq:E}
E_{i,j}=\{x \in \mathbb{F}_q \ : \ \Phi_i(x)=\Phi_j(x)\}, \qquad E^{\cup}=\cup_{i\neq j} E_{ij}.
\end{equation}
By \cite[Theorem 5]{CM2018}, if the permutations $\Phi_i$ ($i=1,2,3$) are such that $E^{\cup}=\mathbb{F}_q$, then the construction of Theorem \ref{Th:Costruzione} gives rise to  Maiorana-McFarland bent functions.

We prove that all the families of triples we provide in our paper satisfy $E^{\cup}\neq \mathbb{F}_q$. 
It should be very interesting to determine whether  the sufficient condition in \cite[Theorem 5]{CM2018} is also necessary for such triples.

\section{Bent functions from triples of permutations satisfying property \texorpdfstring{$(\mathcal{A}_m)$}{Lg}}\label{Sec:Main}

In this section we present our main results concerning the construction of  bent functions via triples of permutations which are monomials (see Theorem \ref{Prop:Fam1}),  binomials and monomials (see Theorem \ref{Prop:Fam2}), monomials and trinomials  (see Theorem \ref{Prop:Fam3}), binomials and quadrinomials (see Theorem \ref{Prop:Fam4}), monomials and quadrinomials (see Theorem \ref{Prop:Fam5}).
{We show that such triples satisfy the hypothesis of Theorem \ref{Th:Costruzione}, which can then be applied to obtain bent functions over $\mathbb{F}_{2^{4m}}\times \mathbb{F}_{2^{4m}}$ (Theorems \ref{Prop:Fam1}, \ref{Prop:Fam2}, \ref{Prop:Fam4}, \ref{Prop:Fam5}) or over $\mathbb{F}_{2^{6m}}\times \mathbb{F}_{2^{6m}}$ (Theorem \ref{Prop:Fam3})}.

We start with a construction involving monomials.

\begin{theorem}\label{Prop:Fam1}
Let $m$ be a positive integer and $\lambda\in \mathbb{F}_{2^{4m}}$ be such that $\lambda^{2^{m}+1}=1$. The triple 
\begin{equation}\label{Family1}
(\Phi_1,\Phi_2,\Phi_3)=(x^{2^{2m}},\lambda x^{2^{m}},\lambda x^{2^{3m}})
\end{equation}
satisfies ($\mathcal{A}_{4m}$) and $E^{\cup}\neq \mathbb{F}_{2^{4m}}$. 
\end{theorem}
\begin{proof}
Clearly $\Phi_i$, $i=1,2,3$, is a permutation of $\mathbb{F}_{2^{4m}}$. It is easily seen that $(\Phi_1^{-1},\Phi_2^{-1},\Phi_3^{-1})=(\Phi_1,\Phi_3,\Phi_2)$ and therefore $\Psi=\Phi_1+\Phi_2+\Phi_3=\Phi_1^{-1}+\Phi_2^{-1}+\Phi_3^{-1}$. Also, $\Psi$ is a permutation {of $\mathbb{F}_{2^{4m}}$} since it is linearized and $\Psi(x)=0$ {with $x\in\mathbb{F}_{2^{4m}}$} if and only if $x=0$. In fact, we have  
$$\lambda x^{2^{m}}+x^{2^{2m}}+\lambda x^{2^{3m}}=0 \iff \left(\frac{1}{\lambda}x +x^{2^{m}}+ \frac{1}{\lambda} x^{2^{2m}}\right)^{2^m}=0\iff x +\lambda x^{2^{m}}+ x^{2^{2m}}=0.$$
Consider now $x \in \mathbb{F}_{2^{4m}}$ such that $x +\lambda x^{2^{m}}+ x^{2^{2m}}=0$. Then    
$$x=x^{2^{4m}}=\left(\left(x +\lambda x^{2^{m}}\right)^{2^m}\right)^{2^m}=\left(x^{2^m} +\frac{1}{\lambda} x^{2^{2m}}\right)^{2^m}=\left(\frac{1}{\lambda}x\right)^{2^m}=\lambda x^{2^m},$$
from which we get $x^{2^{2m}}=0$, that is $x=0$. 

Also, {for any $x\in\mathbb{F}_{2^{4m}}$,}
\begin{eqnarray*}
\Psi(\Psi(x))&=&\lambda \left(\lambda x^{2^{m}}+x^{2^{2m}}+\lambda x^{2^{3m}}\right)^{2^{m}}+ \left(\lambda x^{2^{m}}+x^{2^{2m}}+\lambda x^{2^{3m}}\right)^{2^{2m}}+\lambda \left(\lambda x^{2^{m}}+x^{2^{2m}}+\lambda x^{2^{3m}}\right)^{2^{3m}}\\
&=& x^{2^{2m}}+\lambda x^{2^{3m}}+ x+ \lambda x^{2^{3m}}+x+\lambda x^{2^{m}} + x+\lambda x^{2^{m}}+ x^{2^{2m}}=x,
\end{eqnarray*}
that is $\Psi=\Phi_1+\Phi_2+\Phi_3=\Phi_1^{-1}+\Phi_2^{-1}+\Phi_3^{-1}=\Psi^{-1}$ and property  ($\mathcal{A}_{4m}$) holds. Finally, 
\begin{eqnarray*}
E^{\cup}&=& \{x \in \mathbb{F}_{2^{4m}} \ : \ x^{2^{2m}}=\lambda x^{2^m} \textrm{ or } x^{2^{2m}}=\lambda x^{2^{3m}} \textrm{ or } \lambda x^{2^{m}}=\lambda x^{2^{3m}} \}\\
&=& \{x \in \mathbb{F}_{2^{4m}} \ : \ \lambda x^{2^{m}}= x  \textrm{ or }  x= x^{2^{2m}} \}
\end{eqnarray*}
and $\left| E^{\cup}\right| \leq 2^{2m}+2^{m}< 2^{4m}=|\mathbb{F}_{2^{4m}}|$.

\end{proof}



When using binomials, the following proposition is a key result. 
\begin{proposition}\label{Prop:binomials}
Suppose $\Phi(x)=\alpha x^{2^i}+\beta x^{2^j}$ is a permutation of $\mathbb{F}_{2^n}$ with $i<j<n$. If $\Phi^{-1}(x)$ is a linearized binomial then $n=2m$ and  $j=m+i$  for a positive integer $m$, and 
$\Phi^{-1}(x)=\gamma x^{2^{m-i}}+\delta x^{2^{2m-i}},$
where 
\begin{equation}\label{Eq:alpha,beta}
\gamma=\left(\frac{\beta^{2^{m}}}{\alpha^{2^{m}+1}+\beta^{2^{m}+1}}\right)^{2^{m-i}},\quad \delta=\left(\frac{\alpha}{\alpha^{2^{m}+1}+\beta^{2^{m}+1}}\right)^{2^{m-i}}.
\end{equation}
Also, if $n=2m$, $i<m$, $j=m+i$ and $\alpha,\beta\in\mathbb{F}_{2^n}$ satisfy $\alpha^{2^m+1}\ne\beta^{2^m+1}$, then $\Phi(x)$ is a permutation of $\mathbb{F}_{2^n}$ with $\Phi(x)^{-1}=\gamma x^{2^{m-i}}+\delta x^{2^{2m-i}}$, where $\gamma$ and $\delta$ are given in Equation \eqref{Eq:alpha,beta}.
\end{proposition}

\begin{proof}
Let $\Phi^{-1}(x)=\gamma x^{2^k}+\delta x^{2^\ell}$ with $\gamma,\delta\in\mathbb{F}_{2^n}$ and $k<\ell<n$. Then
$$
\Phi\circ\Phi^{-1}(x)=\alpha\gamma^{2^i}x^{2^{k+i}}+\alpha\delta^{2^i}x^{2^{\ell+i}}+\beta\gamma^{2^j}x^{2^{k+j}}+\beta\delta^{2^j}x^{2^{j+\ell}}\equiv x \pmod{x^{2^n}-x}
$$
only if $k+i\equiv \ell+j\pmod{n}$ and $\ell+i\equiv k+j\pmod n$. This implies that $n$ is even, say $n=2m$; also, $j=m+i$, $\ell=k+m$. Then one between $k+i$ and $k+j$ vanishes modulo $n$.

If $k+i\equiv0\pmod n$, then either $k$ or $i$ is at least $m$, so that either $\ell$ or $j$ is at least $n$, a contradiction to the assumptions. Hence, $k+j\equiv0\pmod{2m}$, that is $k+i\equiv m \pmod{2m}$. This forces $k=m-i$ and $\ell=2m-i$. Then
$$
\Phi\circ\Phi^{-1}(x)=\left(\alpha\gamma^{2^i}+\beta\delta^{2^{m+i}}\right)x^m + \left(\alpha\delta^{2^i}+\beta\gamma^{2^{m+i}}\right)x
$$
equals $x\pmod{x^{2^n}-x}$ if and only if \begin{equation}\label{Prima}\alpha\gamma^{2^i}+\beta\delta^{2^{m+i}}=0, \quad\alpha\delta^{2^i}+\beta\gamma^{2^{m+i}}=1.\end{equation}
From $\Phi^{-1}\circ\Phi(x)\equiv x\pmod{x^{2^n}-x}$ it follows also \begin{equation}\label{Seconda}\gamma\alpha^{2^{m-i}}+\delta\beta^{2^{2m-i}}=0,\quad \gamma\beta^{2^{m-i}}+\delta\alpha^{2^{2m-i}}=1. \end{equation}
By the uniqueness of $\Phi^{-1}(x)$, we have that $\alpha^{2^m+1}\ne\beta^{2^m+1}$ and $\gamma,\delta$ are as in Equation \eqref{Eq:alpha,beta} and satisfy both Equations \eqref{Prima} and \eqref{Seconda}.

If $n=2m$, $i<m$, $j=m+i$, $\alpha,\beta\in\mathbb{F}_{2^n}$ satisfy $\alpha^{2^m+1}\ne\beta^{2^m+1}$, and $\gamma,\delta$ as in Equation \eqref{Eq:alpha,beta}, then it is readily seen that $\Phi(x)$ is a permutation and $\Phi^{-1}(x)=\gamma x^{2^{m-i}}+\delta x^{2^{2m-i}}$.
\end{proof}


As an application, we produce triples of permutation polynomials satisfying property ($\mathcal{A}_n$) using binomials. 
\begin{theorem}\label{Prop:Fam2}
Let $m$ be a positive integer,  $\alpha\in \mathbb{F}_{2^{m}}$. 
The triple 
\begin{equation}\label{Family2}
(\Phi_1,\Phi_2,\Phi_3)=(x^{2^{2m}},\alpha x^{2^m}+(\alpha+1) x^{2^{3m}},(\alpha+1) x^{2^m}+\alpha x^{2^{3m}})
\end{equation} 
satisfies ($\mathcal{A}_{4m}$) and $E^{\cup}\neq \mathbb{F}_{2^{4m}}$.
\end{theorem}

\begin{proof}
It is easily seen that $\Phi_2$ and $\Phi_3$ are bijections and $\Phi_2^{-1}=\Phi_3$, by Proposition \ref{Prop:binomials}. Also, 
$$\Psi(x)=\Phi_1(x)+\Phi_2(x)+\Phi_3(x)=x^{2^m}+x^{2^{2m}}+x^{2^{3m}}$$
is an involution and
$$\Phi_1^{-1}(x)+\Phi_2^{-1}(x)+\Phi_3^{-1}(x)=\Phi_1(x)+\Phi_2(x)+\Phi_3(x)=\Psi(x)=\Psi^{-1}(x).$$
Therefore, $(\Phi_1,\Phi_2,\Phi_3)$ satisfies ($\mathcal{A}_{4m}$). Arguing as in Proposition \ref{Prop:Fam1} it is easily seen that $\left| E^{\cup}\right| \leq 2^{2m}+2^{m} < 2^{4m}=|\mathbb{F}_{2^{4m}}|$ and the claim follows.
\end{proof}


\begin{remark}
It is worth noting that binomials have been already used in \cite[Theorem 7]{CM2018}, where the authors use the classification given in \cite[Proposition 5]{CMS2015} of the linear binomial involutions of $\mathbb{F}_{2^m}$. In Theorem \ref{Prop:Fam2} we do not use involutions. 
\end{remark}

Trinomials are investigated in the next proposition. 
\begin{proposition}\label{Prop:trinomials}
Let $m$ be a positive integer. Two trinomials \begin{eqnarray*}
\Phi(x)&=&\alpha x^{2^m}+\beta x^{2^{3m}}+\gamma x^{2^{5m}}\in \mathbb{F}_{2^{6m}}[x],\\
\Psi(x)&=&\overline{\alpha} x^{2^m}+\overline{\beta} x^{2^{3m}}+\overline{\gamma} x^{2^{5m}}\in \mathbb{F}_{2^{6m}}[x]
\end{eqnarray*}
satisfy $\Phi \circ \Psi(x)=\Psi \circ \Phi(x)=x$ for each $x\in \mathbb{F}_{2^{6m}}$ if and only if
\begin{equation}\label{EquazioniTrinomi}
    \left\{
    \begin{array}{l}
    \overline{\alpha}\gamma^{2^m}+\overline{\beta}\beta^{2^{3m}}+\overline{\gamma}\alpha^{2^{5m}}=1\\
    \overline{\alpha}\alpha^{2^m}+\overline{\beta}\gamma^{2^{3m}}+\overline{\gamma}\beta^{2^{5m}}=0\\
    \overline{\alpha}\beta^{2^m}+\overline{\beta}\alpha^{2^{3m}}+\overline{\gamma}\gamma^{2^{5m}}=0
    \end{array}
    \right.
\end{equation}
and 
\begin{equation}\label{DeterminanteTrinomi}
\det\left( \begin{array}{ccc}
\gamma^{2^m}&\beta^{2^{3m}}&\alpha^{2^{5m}}\\
    \alpha^{2^m}&\gamma^{2^{3m}}&\beta^{2^{5m}}\\
    \beta^{2^m}&\alpha^{2^{3m}}&\gamma^{2^{5m}}\\
\end{array}\right)\neq 0.
\end{equation}
\end{proposition}
\proof
By direct checking $\Psi\circ \Phi(x)$ equals
$$\left(\overline{\alpha}\gamma^{2^m}+\overline{\beta}\beta^{2^{3m}}+\overline{\gamma}\alpha^{2^{5m}}\right)x+
    \left(\overline{\alpha}\alpha^{2^m}+\overline{\beta}\gamma^{2^{3m}}+\overline{\gamma}\beta^{2^{5m}}\right) x^{2^{2m}}+
    \left(\overline{\alpha}\beta^{2^m}+\overline{\beta}\alpha^{2^{3m}}+\overline{\gamma}\gamma^{2^{5m}}\right) x^{2^{4m}}.$$
So, if $\Psi=\Phi^{-1}$, then $\Psi\circ \Phi(x)=x$ for all $x \in \mathbb{F}_{2^{6m}}$,  and  \eqref{EquazioniTrinomi} holds. Thus,  \eqref{EquazioniTrinomi} has a unique solution $(\overline{\alpha},\overline{\beta},\overline{\gamma})$ and \eqref{DeterminanteTrinomi} holds. On the other hand, if  $(\overline{\alpha},\overline{\beta},\overline{\gamma})$ is a solution of \eqref{EquazioniTrinomi} and \eqref{DeterminanteTrinomi}, then  $\Phi\circ \Psi(x)=x$ for each $x\in \mathbb{F}_{2^{6m}}$ and $\Psi=\Phi^{-1}$.
\endproof

Particular choices of permutation trinomials give rise to the following triples of functions satisfying property ($\mathcal{A}_n$). 
\begin{theorem}\label{Prop:Fam3}
Let $m$ be a positive, $\alpha,\beta \in \mathbb{F}_{2^{2m}}$ be such that either 
\begin{enumerate}[(i)]
    \item \label{I} $\alpha=\beta$, $\alpha^{2^m}+\alpha^{2^m-1}+1=0$, or
    \item\label{II} $\beta=\alpha+1$, $\alpha^{2^m+1}=1$, $\alpha\neq 1$.
\end{enumerate}
The triple 
\begin{equation}\label{Family3}
(\Phi_1,\Phi_2,\Phi_3)=(x^{2^{m}},\alpha x^{2^m}+\beta x^{2^{3m}}+(\alpha+1) x^{2^{5m}},(\alpha+1) x^{2^m}+\beta x^{2^{3m}}+\alpha x^{2^{5m}})
\end{equation} 
satisfies ($\mathcal{A}_{6m}$) and $E^{\cup}\neq \mathbb{F}_{2^{6m}}$.
\end{theorem}
\proof
By direct checking the $6$-tuple $(\alpha,\beta,\gamma,\overline{\alpha},\overline{\beta},\overline{\gamma})=(\alpha,\beta, \alpha+1,\alpha+1,\beta,\alpha)$, where $\alpha$ and $\beta$ are as in \eqref{I} or \eqref{II}, satisfies Conditions \eqref{EquazioniTrinomi} and \eqref{DeterminanteTrinomi} therefore  $\Phi_2^{-1}=\Phi_3$. Also, 
$$\Psi(x)=\Phi_1(x)+\Phi_2(x)+\Phi_3(x)=x^{2^m}+x^{2^m}+x^{2^{5m}}=x^{2^{5m}}$$
is a bijection and 
$$\Phi_1^{-1}(x)+\Phi_2^{-1}(x)+\Phi_3^{-1}(x)=x^{2^m}=\Psi^{-1}(x)$$
{over $\mathbb{F}_{2^{6m}}$.}
Hence, $(\Phi_1,\Phi_2,\Phi_3)$ satisfies ($\mathcal{A}_m$). Arguing as in Proposition, \ref{Prop:Fam1} $\left| E^{\cup}\right| \leq 2^{4m}+2^{4m}+2^{4m} < 2^{6m}=|\mathbb{F}_{2^{6m}}|$ and the claim follows.
\endproof

The last part of this section involves permutation quadrinomials. 
\begin{lemma}\label{involution}
Let $q$ be a power of $2$, $C\in\mathbb{F}_q$, $D\in\mathbb{F}_{q^2}$. Then the polynomial
$f(x)=Cx+Dx^q+(C+1)x^{q^2}+Dx^{q^3}$
is an involution over $\mathbb{F}_{q^4}$.
\end{lemma}

\begin{proof}
By direct computation,
$$ f^2(x)=C\left(Cx+Dx^q+(C+1)x^{q^2}+Dx^{q^3}\right)+D\left(Cx^q+D^qx^{q^2}+(C+1)x^{q^3}+D^q x\right) $$
$$ +(C+1)\left(Cx^{q^2}+Dx^{q^3}(C+1)x+Dx^q\right)+D\left(Cx^{q^3}+D^q x+(C+1)x^q+D^q x^{q^2}\right).$$
Since $q$ is even, this implies $f^2(x)=x$.
\end{proof}

\begin{theorem}\label{Prop:Fam4}
Let $\alpha\in\mathbb{F}_{2^m}$ and $\beta\in\mathbb{F}_{2^{2m}}$ be such that $\beta^{2^m+1}=\alpha^2+1$.
Then the triple
$$(\Phi_1,\Phi_2,\Phi_3)=\left( \alpha x+\beta x^{2^m}, \alpha^2 x + \beta^{2^m+1}x^{2^{2m}} , \alpha^3 x + \alpha^2\beta x^{2^m} + \alpha\beta^{2^m+1} x^{2^{2m}} + \beta^{2^m+2}x^{2^{3m}} \right)$$
satisfies ($\mathcal{A}_{4m}$) and $E^{\cup}\neq \mathbb{F}_{2^{4m}}$.
\end{theorem}

\begin{proof}
By direct checking, $\Phi_i(x)=\Phi_1^i(x)$ for all $x\in\mathbb{F}_{2^{4m}}$, $i=2,3$, and $\Phi_2(x)$ is an involution over $\mathbb{F}_{2^{4m}}$ by Lemma \ref{involution}, so that $\Phi_3(x)=\Phi_1^{-1}(x)$ over $\mathbb{F}_{2^{4m}}$.
Also, for all $x\in\mathbb{F}_{2^{4m}}$, we have
$$ \Psi(x)= \Phi_1(x)+\Phi_2(x)+\Phi_3(x) = (\alpha+\alpha^2+\alpha^3)x + (\alpha^2+1)\beta x^{2^m} + (1+\alpha+\alpha^2+\alpha^3)x^{2^{2m}}+(\alpha^2+1)\beta x^{2^{3m}} $$
and
$$ \Phi_1^{-1}(x)+\Phi_2^{-1}(x)+\Phi_3^{-1}(x)=\Psi(x)=\Psi^{-1}(x). $$
Thus, $(\Phi_1,\Phi_2,\Phi_3)$ satisfies ($\mathcal{A}_{4m}$) and $|E^{\cup}|\leq 2^m + 2^{2m}+2^{3m}<2^{4m}=| \mathbb{F}_{2^{4m}}|$. The claim follows.
\end{proof}

\begin{theorem}\label{Prop:Fam5}
Let $C\in\mathbb{F}_{2^m}$, $D\in\mathbb{F}_{2^{2m}}$, and $\lambda\in\mathbb{F}_{2^{2m}}$ be such that $\lambda^{2^m+1}=1$. Then the triple
$$(\Phi_1,\Phi_2,\Phi_3)=\left( \lambda x^{2^m}, \lambda x^{2^{3m}} , Cx+Dx^{2^m}+(C+1)x^{2^{2m}}+Dx^{2^{3m}} \right)$$
satisfies ($\mathcal{A}_{4m}$) and $E^{\cup}\neq \mathbb{F}_{2^{4m}}$.
\end{theorem}

\begin{proof}
From $\lambda^{2^m+1}=1$ follows $\Phi_2(x)=\Phi_1^{-1}(x)$ for all $x\in\mathbb{F}_{2^{4m}}$. Also, $\Phi_3(x)$ is an involution over $\mathbb{F}_{2^{4m}}$ by Lemma \ref{involution}. Then 
$$ \Psi(x)= \Phi_1(x)+\Phi_2(x)+\Phi_3(x) = Cx+(D+\lambda)x^{2^m}+(C+1)x^{2^{2m}}+(D+\lambda)x^{2^{3m}} $$
is an involution over $\mathbb{F}_{2^{4m}}$ by Lemma \ref{involution}, and
$$ \Psi^{-1}(x)=\Phi_1^{-1}(x)+\Phi_2^{-1}(x)+\Phi_3^{-1}(x)=\Psi(x).$$
Thus, $(\Phi_1,\Phi_2,\Phi_3)$ satisfies ($\mathcal{A}_{4m}$).

If $m \geq 2$, then $|E^{\cup}|\leq 2^m + 2^{3m}+2^{3m}<2^{4m}=| \mathbb{F}_{2^{4m}}|$.
Suppose $m=1$. Then
$ E_{1,2}=\{0,1,\lambda,\lambda^2\}$ has size at most $4$. Also, the elements of $E_{1,3}$ are the roots in $\mathbb{F}_{16}$ of $f(T)=DT^8+(C+1)T^4+(D+\lambda)T^2+CT$. By direct checking, $\gcd(f(T),T^{16}+T)$ has order at most $4$. Since $\{0,\lambda^2\}\subseteq E_{1,2}\cap E_{1,3}$, the size of $E_{1,2}\cup E_{1,3}$ is at most $6$.
Thus, $E^{\cup}$ has size at most $6+8<16$ and the claim follows.
\end{proof}



\section{Acknowledgments}
The research of D. Bartoli, M. Montanucci, G. Zini was partially supported  by the Italian National Group for Algebraic and Geometric Structures and their Applications (GNSAGA - INdAM).



\begin{thebibliography}{99}

\bibitem{2} R. Calderbank and W. M. Kantor, ``The geometry of two-weight codes,"
Bull. London Math. Soc., vol. 18, no. 2, pp. 97--122, 1986.


\bibitem{Carlet2006}
 C. Carlet, ``On bent and highly nonlinear balanced/resilient functions and their algebraic immunities," in: \emph{Applied Algebra, Algebraic Algorithms and Error-Correcting Codes,} Berlin, Germany: Springer, 2006, pp. 1--28.

\bibitem{8} C. Carlet, ``Boolean functions for cryptography and error correcting
codes," in \emph{Boolean Models and Methods in Mathematics, Computer Science, and Engineering}, Y. Crama and P. L. Hammer, Eds. Cambridge, U.K.: Cambridge Univ. Press, 2010, pp. 257--397.


\bibitem{Carlet2014}
 C. Carlet, ``Open problems on binary bent functions," in \emph{Open Problems in Mathematics and Computational Science}, Cham, Switzerland: Springer, 2014, pp. 203--241.

\bibitem{CM2011}
 C. Carlet and  S. Mesnager, ``On Dillon's class H of bent functions, Niho bent functions and O-polynomials," \emph{J. Combinat. Theory A}, vol. 118, no. 8, pp. 2392--2410, 2011.

\bibitem{CM2016}
 C. Carlet and  S. Mesnager, ``Four decades of research on bent functions," \emph{Des. Codes Cryptograph.}, vol. 78, no. 1, pp.  5--50, 2016.



\bibitem{CMS2015}
 P. Charpin, S. Mesnager, and  S. Sarkar, ``On involutions of finite fields," in  \emph{Proc. IEEE Int. Symp. Inf. Theory (ISIT)},  2015, pp. 186--190.

\bibitem{CMS2016}
P. Charpin, S. Mesnager, and  S. Sarkar, ``Involutions over the Galois field $\mathbb{F}_{2^n}$," \emph{IEEE Trans. Inf. Theory}, vol. 62, no. 4, pp. 2266--2276, 2016.

\bibitem{14} G. Cohen, I. Honkala, S. Litsyn, and A. Lobstein, \emph{Covering Codes} Amsterdam, The Netherlands: North Holland, 1997.


\bibitem{CM2018}
R.S. Coulter and  S. Mesnager, ``Bent Functions From Involutions Over $\mathbb{F}_{2^n}$," \emph{IEEE Trans. Inf. Theory}, vol. 64, no. 4, pp. 2979--2986, 2018.


\bibitem{Dillon1974}
J.F. Dillon, ``Elementary Hadamard difference sets," Ph.D. disseration, Univ. Maryland, College Park, MD, USA, 1974.

\bibitem{Kyureghyan2011}
G.M. Kyureghyan, ``Constructing permutations of finite fields via linear translators," \emph{J. Combinat. Theory A}, vol.  118, no. 3, pp. 1052--1061, 2011.

\bibitem{McFarland1973}
R.L. McFarland, ``A family of difference sets in non-cyclic groups," \emph{J. Combinat. Theory A}, vol.  15, no 1, pp. 1--10, 1973.

\bibitem{Mesnager2014}
S. Mesnager, ``Several new infinite families of bent functions and their duals," \emph{IEEE Trans. Inf. Theory}, vol.  60, no 7, pp. 4397--4407, 2014.

\bibitem{Mesnager2016_a}
S. Mesnager, ``Further constructions of infinite families of bent functions from new permutations and their duals,"  \emph{Cryptograph. Commun.}, vol.  8, no. 2, pp. 229--246, 2016.

\bibitem{Mesnager2016_b}
S. Mesnager, \emph{Bent Functions: Fundamentals and Results},  New York, NY, USA: Springer-Verlag, 2016.

\bibitem{Mesnager2016_c}
S. Mesnager, ``On constructions of bent functions from involutions," in \emph{Proc. ISIT}, Jul. 2016, pp. 110--114.

\bibitem{MCM2015}
S. Mesnager, G. Cohen, and D. Madore, ``On existence (based on an arithmetical problem) and constructions of bent functions," in \emph{Proc. 15th Int. Conf. Cryptograph. Coding}, 2015, pp. 3--19.

\bibitem{30} J. D. Olsen, R. A. Scholtz, and L. R. Welch, ``Bent-function sequences,"
\emph{IEEE Trans. Inf. Theory}, vol. 28, no. 6, pp. 858--864, Nov. 1982.

\bibitem{31} A. Pott, Y. Tan, and T. Feng, ``Strongly regular graphs associated with
ternary bent functions," \emph{J. Combinat. Theory, Ser. A}, vol. 117, no. 6, pp. 668--682, 2010.



\bibitem{Rothaus1976}
O.S. Rothaus, ``On `bent' functions," \emph{J. Combinat. Theory A}, vol. 20, no. 3, pp.  300--305, 1976.




\end{thebibliography}
\end{document}